\documentclass[preprint,11pt]{elsarticle}

\usepackage{amsfonts}
\usepackage{amsmath}
\usepackage{amssymb}
\usepackage{latexsym}
\usepackage{graphicx}
\usepackage{ntheorem}
\usepackage[usenames]{color}

\newtheorem{theo}{Theorem}[section]
\newtheorem{theorem}[theo]{Theorem}

\newtheorem{lemma}[theo]{Lemma}
\newtheorem{definition}[theo]{Definition}

\newtheorem{remark}[theo]{Remark}
\theoremstyle{empty}
\theorembodyfont{\rmfamily}
\newtheorem{refproof}{theorem}

\journal{Applied Mathematics and Computation}

\begin{document}

\begin{frontmatter}

\title{On the spectrum of the normalized Laplacian of iterated triangulations of graphs}

\author[label01,label3]{Pinchen Xie}
\author[label02,label3]{Zhongzhi Zhang}
\ead{zhangzz@fudan.edu.cn}
\author[label04]{Francesc Comellas}
\ead{comellas@ma4.upc.edu}
\address[label01]{Department of Physics, Fudan
University, Shanghai 200433, China}
\address[label02]{School of Computer Science, Fudan
University, Shanghai 200433, China}
\address[label3]{Shanghai Key Laboratory of Intelligent Information
Processing, Fudan University, Shanghai 200433, China}
\address[label04]{Department of Applied Mathematics IV,
Universitat Polit\`ecnica de Catalunya, 08034 Barcelona, Catalonia, Spain}

\begin{abstract}

The eigenvalues of the normalized Laplacian of a graph provide information on its topological and structural characteristics and also on some relevant dynamical aspects, specifically in relation to random walks.  In this paper we determine the spectra of the normalized Laplacian of iterated triangulations of a generic simple connected graph. As an application, we also find closed-forms for their multiplicative degree-Kirchhoff index, Kemeny's constant and number of spanning trees.

\end{abstract}


\end{frontmatter}

\section{Introduction}

Important structural and dynamical properties of networked systems can be obtained from the eigenvalues and eigenvectors of  matrices associated to their graph representations. The spectra of the adjacency, Laplacian and normalized Laplacian matrices of a graph provide information on the diameter, degree distribution, community structure, paths of a given length,  local clustering, total number of links, number of spanning trees and many more invariants~\cite{CaFaKi10,Ch97,MeBa15,NaNe12,Mi10}. Also, dynamic aspects of a network, such as its synchronizability and random walk properties, can be obtained from the eigenvalues of the Laplacian and normalized Laplacian matrices from which it is possible to calculate some interesting graph invariants like the Estrada index and the Laplacian energy~\cite{BiCoPaTo13,BrHa12,DaGuCeZh13,GoRo01,Kh13}.

In the last years, there has been a particular interest in the study of the eigenvalues and eigenvectors of the normalized Laplacian matrix, since many measures for random walks on a network are linked to them. These include the hitting time, mixing time and Kemeny's constant which can be considered a measure of the efficiency of navigation on the network, see~\cite{HuLi15,KeSn76,LeLo02,Lo93,ZhShCh13}.

This paper is organized as follows. First, in Sec~\ref{pre}, we recall an operation called triangulation that can be applied to any simple connected graph. In Sec~\ref{sec3}, we determine  the spectra of  the normalized Laplacians of iterated triangulations of any simple connected graph and we discuss their structure.
Finally, in  Sec~\ref{app}, we  use the results found to calculate three significant invariants for an iterated triangulation of a graph: the multiplicative degree-Kirchhoff index, its Kemeny's constant and  the number of spanning trees.

\section{Preliminaries}\label{pre}

Let $G(V,E)$ be any simple connected graph with vertex set $V$ and edge set $E$.
Let $N=N_0=|V|$ denote the number of vertices of $G$ and $E_0=|E|$ its number of edges.

We consider a graph operation known in different contexts as Henneberg 1 move~\cite{He11}, $0$-extension~\cite{NiRo14}, vertex addition~\cite[p. 63]{CvDoSa80} and triangulation~\cite{RoKl14}.  Here we will use the latter name as it is more common in  recent literature.

\begin{definition}\label{define}.
The triangulation of $G$, denoted by $\tau(G)$, is the graph obtained by adding a new vertex corresponding to each edge  of G and by joining each new vertex  to the end vertices of the edge corresponding to it .
\end{definition}

We denote $\tau^0(G)=G$. The $n$-triangulation of $G$ is obtained through the iteration
$\tau^n(G)=\tau(\tau^{n-1}(G))$ and $N_n$ and $E_n$  will be the total number of vertices and edges of $\tau^n(G)$.  	
		
From this definition clearly  $E_n=3E_{n-1}$ and $N_n=N_{n-1}+E_{n-1}$ and thus we have:
\begin{equation}
\  E_n=3^{n}E_0 ,\quad \quad \ N_n=N_0+\frac{3^n-1}{2}E_0.
\end{equation}

Figure~\ref{Fig.1}  shows an example of iterated triangulations  where the initial graph is $K_3$ or the triangle graph. In this particular case, the resulting graph is known as  scale-free pseudofractal graph~\cite{DoGoMe02} that exhibits a scale-free and small-world topology. The study of its structural and dynamic properties has produced an abundant literature, since it is a good deterministic model for many real-life networks, see~\cite{CoHa10,CoFeRa04,Ne10,ZhRoZh07,ZhZhCh07} and references therein.
\begin{figure}
\centering
\includegraphics[width=0.4\textwidth]{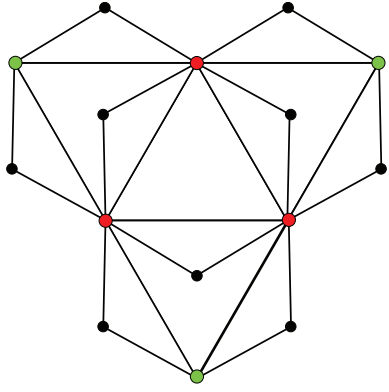}
\caption{An example of $\tau^2(G)$. Red vertices denote the initial $K_3$ vertices while green and black vertices have been introduced  to obtain $\tau(K_3)$ and $\tau^2(K_3)$, respectively.}
 \label{Fig.1}
 \end{figure}

The triangulation operator $\tau$ turns each edge of a graph into a triangle, i.e. a cycle of length $3$.
As we will see in the next section, this property has a relevant impact on the structure of the  spectrum of the normalized Laplacian matrix of $\tau^n(G)$.

We label the vertices of $\tau^n(G)$ from $1$ to $N_n$.
Let $d_i$ be the degree of vertex $i$ of $\tau^n(G)$, then   $D_n={\rm diag}(d_1, d_2, \cdots, d_{N_n})$ denotes
its diagonal degree matrix and $A_n$ the adjacency matrix, defined as a matrix with the $(i,j)$-entry  equal to $1$ if vertices $i$ and $j$ are adjacent and $0$ otherwise.

The Laplacian matrix of $\tau^n(G)$ is $L_n=D_n-A_n$. The probability transition matrix for random walks on  $\tau^n(G)$ or Markov matrix is  $M_n=D_n^{-1}A_n$. $M_n$ can be normalized to obtain a symmetric matrix $P_n$.
\begin{equation}
P_n=D_n^{-\frac{1}{2}} A_n D_n^{-\frac{1}{2}}=D_n ^{\frac{1}{2}}M_n D_n^{-\frac{1}{2}}\,.
\end{equation}
The $(i,j)$th entry of $P_n$ is $P_n(i,j)=\frac{A_n(i,j)}{\sqrt{d_i d_j}}$.
Matrices $P_n$ and $M_n$ share the same eigenvalue spectrum.

\begin{definition}
The normalized Laplacian matrix of $\tau^n(G)$ is
\begin{equation}\label{LP}
{\cal L}_n=I-D_n ^{\frac{1}{2}}M_n D_n^{-\frac{1}{2}}=I-P_n,
\end{equation}
where $I$ is the identity matrix with the same order as $P_n$.
\end{definition}

We denote the spectrum of  ${\cal L}_n$ by $\sigma_n=\left\{\lambda_1^{(n)},\lambda_2^{(n)},\cdots,\lambda_{N_n}^{(n)}\right\}$ where $0=\lambda_1^{(n)}<\lambda_2^{(n)}\leqslant\cdots\leqslant\lambda_{N_n-1}^{(n)}\leqslant\lambda_{N_n}^{(n)}\leqslant2$. The zero eigenvalue is unique, due to the existence of a stationary distribution for random walks. We denote the multiplicity of $\lambda_i$ as $m_{{\cal L}_n}(\lambda_i)$. 
Eq.~\eqref{LP} shows a one-to-one correspondence between the spectra of ${\cal L}_n$  and $P_n$.

The  spectrum of  the normalized Laplacian matrix of a graph can provide us with important structural and dynamical information about the graph, see for example~\cite{Bu16,Ch97}.

 \begin{definition}
If we replace each edge of a simple connected graph $G$ by a unit resistor, we obtain an electrical network $G^*$ associated with $G$. The resistance distance $r_{ij}$ between vertices $i$ and $j$ of $G$ is equal to the effective resistance between the two corresponding vertices of $G^*$~\cite{KlRa93}.
\end{definition}

\begin{definition}
The multiplicative degree-Kirchhoff index of $G$ is defined, as in~\cite{ChZh07}:
\begin{equation}
{K\! f}^*(G)=\sum_{i<j}d_i d_j r_{ij}, \quad  i,j=1,2,\cdots N.
\end{equation}
\end{definition}

This index is different from the classical Kirchhoff index~\cite{BoBaLiKl94}, ${K\! f}(G)=\sum_{i<j}r_{ij}$, as it considers also the degree distribution of the graph. See also~\cite{FeGuYu13,Li13,Pa13} for recents results on these and related indices. For a $k$-regular graph ${K\! f}^* = k^2 {K\! f}$.

It is known~\cite{ChZh07} that ${K\! f}^*(G)$  can be expressed in terms of  the spectrum $\sigma_0=\{\lambda_1,\cdots,\lambda_{N_0}\}$ of the normalized Laplacian matrix  ${\cal L}_0$ of $G$. Thus
\begin{equation}
	{K\! f}^*(G)=2E_0\sum_{k=2}^{N_0}\frac{1}{\lambda_k},
\end{equation}
where $0=\lambda_1<\lambda_2\leqslant\cdots\leqslant\lambda_{N_0}\leqslant 2$.

For $n\geqslant 0$, we can obtain the degree-Kirchhoff index of the $n$-triangulation of $G$ as:
\begin{equation}\label{kirchhoff}
	{K\! f}^*(\tau^n(G))=2E_n\sum_{k=2}^{N_n}\frac{1}{\lambda^{(n)}_k},
\end{equation}
where $0=\lambda_1^{(n)}<\lambda_2^{(n)}\leqslant\cdots\leqslant\lambda_{N_n-1}^{(n)}\leqslant\lambda_{N_n}^{(n)}\leqslant2$ are the eigenvalues of ${\cal L}_n$.

The knowledge of $\sigma_n$ also allows the calculation of an important graph invariant associated with random walks.

\begin{definition}
	Given a graph $G$, the Kemeny's constant $K(G)$, or average hitting time, is the expected number of steps required for the transition from a starting vertex $i$ to a destination vertex, which is chosen randomly according to a stationary distribution of  unbiased random walks on $G$ (see~\cite{Hu14} for more details).
\end{definition}

$K(G)$ is a constant which is independent of the selection of the starting vertex $i$~\cite{LeLo02}.
Moreover,  the Kemeny's constant can be computed as the sum of all reciprocal eigenvalues of  the
normalized Laplacian of $G$, except $1/\lambda_1$, see~\cite{Bu16}.

Therefore, given the spectrum $\sigma_n$ of $\tau^n(G)$ we can write:
\begin{equation}\label{kem}
	K(\tau^n(G))=\sum_{k=2}^{N_n}\frac{1}{\lambda^{(n)}_k}.
\end{equation}
Moreover,  we have:
\begin{equation}\label{KH}
	{K\! f}^*(\tau^n(G))=2E_n\cdot K(\tau^n(G)).
\end{equation}
This equation reflects the fact that, for a connected graph, the resistance distance can be related to random walks~\cite{Pa01}.

The last graph invariant considered  is the number of spanning trees of a graph $G$.
A spanning tree is a subgraph of $G$ that includes all the vertices of $G$ and is a tree.
By using a result from Chung~\cite{Ch97} we can find the total number of spanning trees of $\tau^n(G)$
in terms of  $\sigma_n$, its  normalized Laplacian spectrum,  and the degrees of all the vertices:
\begin{equation}\label{ST01}
N_{\rm st}^{(n)}=\frac{\displaystyle \prod_{i=1}^{N_{n}}
d_i\prod_{i=2}^{N_n}\lambda_i^{(n)}}{\displaystyle \sum_{i=1}^{N_{n}}d_i}\,.
\end{equation}
In the next section we give a closed analytical expression for this invariant for any value of $n\geqslant 0$.

\section{The normalized Laplacian spectrum of the triangulation graph $\tau^n(G)$ }\label{sec3}

In this section we  find an exact analytical expression for the spectrum $\sigma_n$ of ${\cal L}_n$, the normalized  Laplacian  of the $n$-triangulation graph $\tau^n(G)$. We show that this spectrum can be  obtained iteratively from the spectrum of any simple connected graph $G$.

\begin{lemma}\label{eig1}
Let $\lambda$ be any  eigenvalue of  ${\cal L}_n$,  $n> 0$,  such that $\lambda \neq 1$ and $\lambda \neq \frac{3}{2}$. Then $2\lambda$ is an eigenvalue of ${\cal L}_{n-1}$ and its multiplicity, denoted $m_{{\cal L}_{n-1}}(2\lambda)$,
is the same than the multiplicity $m_{{\cal L}_{n}}(\lambda)$ of the eigenvalue $\lambda$ of ${\cal L}_n$. Moreover,  $m_{{\cal L}_{n}}(\frac{3}{2})=N_{n-1}$.
\end{lemma}
\begin{proof}
Let ${V}_{\rm new}^n$ be the set of all the newly added vertices  in $\tau^n(G)$ and  ${V}_{\rm old}^n$ be its complement.  This is also equivalent to saying that ${V}_{\rm old}^n$ contains all the vertices inherited from  $\tau^{n-1}(G)$.  For convenience, in the following when we refer to any vertex of ${V}_{\rm old}^n$,  we also refer to  the corresponding vertex of $\tau^{n-1}(G)$.

Let $\psi=(\psi_{1},\psi_{2},\ldots,\psi_{N_n})^\top$ be an eigenvector with respect to the eigenvalue $\lambda$ of  ${\cal L}_n$. Then, as ${\cal L}_n=I-P_n$,  $\psi$ is also an eigenvector of $P_n$ corresponding to the eigenvalue $(1-\lambda)$. Hence
\begin{equation}\label{F0}
P_n\psi=(1-\lambda)\psi\,.
\end{equation}
For any vertex $i\in{V}_{\rm old}^n$, let $ {N}\subset{V}_{\rm new}^n$ denote the set of the new  neighbors of vertex $i$ in $\tau^n(G)$ and $ {N}'\subset{V}_{\rm old}^n$  the set of its old neighbors. There exists, from the definition of triangulation,  a bijection between  ${N}$ and ${N}'$.

Eq.~\eqref{F0} can be written as
\begin{equation}\label{F1}
\sum_{j=1}^{N_n}P_n(i,j)\psi_j=(1-\lambda)\psi_i.
\end{equation}
If $d_n(i)$ denotes the degree of vertex $i$ of $\tau^n(G)$, this equation  leads to
\begin{equation}\label{F2}
\begin{split}
(1-\lambda)\psi_i &=\sum_{\  j\in {N}}\frac{1}{\sqrt{d_n(i)d_n(j)}}\psi_j+\sum_{\  j'\in {N'}}\frac{1}{\sqrt{d_n(i)d_n(j')}}\psi_j'\\
	&=\sum_{\  j\in {N}}\frac{1}{2\sqrt{d_{n-1}(i)}}\psi_j+\sum_{\  j'\in {N'}}\frac{1}{2\sqrt{d_{n-1}(i)d_{n-1}(j')}}\psi_j'.
	\end{split}
\end{equation}
For any $vertex\ j\in {N}$, we have a similar relationship
\begin{equation}\label{F3}
	(1-\lambda)\psi_j=\sum_{k=1}^{N_n}P_n(j,k)\psi_k=\frac{1}{2\sqrt{d_{n-1}(i)}}\psi_i+\frac{1}{2\sqrt{d_{n-1}(j')}}\psi_{j'},
\end{equation}
where vertex $ j'\in {N}'$ is the other neighbor of vertex $ j$ in  $\tau^n(G)$ (excluding the vertex $i$).
Combining Eq.~\eqref{F2} and Eq.~\eqref{F3} we have:
\begin{equation}\label{F4}
\begin{split}
	(1-\lambda)^2\psi_i =&\frac{1}{2\sqrt{d_{n-1}(i)}}\cdot\sum_{\  j'\in {N}'}\left(\frac{1}{2\sqrt{d_{n-1}(i)}}\psi_i+\frac{1}{2\sqrt{d_{n-1}(j')}}\psi_{j'}\right) \\
	 &+\frac{1-\lambda}{2}\cdot\sum_{\  j'\in {N}}\frac{1}{\sqrt{d_{n-1}(i)d_{n-1}(j')}}\psi_j'\\
	=&\frac{\psi_i}{4}+\frac{3-2\lambda}{4}\cdot\sum_{\  j'\in {N'}}\frac{1}{\sqrt{d_{n-1}(i)d_{n-1}(j')}}\psi_j'.
\end{split}
\end{equation}
Therefore,
\begin{equation}\label{F5}
	(1-2\lambda)\psi_i=\sum_{\  j'\in {N}'}\frac{1}{\sqrt{d_{n-1}(i)d_{n-1}(j')}}\psi_{j'}
\end{equation}
holds for $\lambda\neq\frac{3}{2}$.

Eq.~\eqref{F5}  shows  that $(1-2\lambda)$ is an eigenvalue of $P_{n-1}$. Thus $2\lambda$ is an eigenvalue of ${\cal L}_{n-1}$ and $\psi_o=(\psi_i)_{i\in {V}_{\rm old}^n}^\top$ one associated eigenvector.  $\psi$ can be totally determined by $\psi_o$ using  Eq.~\eqref{F3} and $m_{{\cal L}_{n-1}}(2\lambda) \geqslant m_{{\cal L}_n}(\lambda)$.

Suppose now that  $m_{{\cal L}_{n-1}}(2\lambda) > m_{{\cal L}_n}(\lambda)$.
This means that ithere should exist an extra eigenvector $\psi_e$ associated to $2\lambda$ without a corresponding eigenvector in ${\cal L}_n$. But Eq.~\eqref{F3} provides  $\psi_e$ with an associated eigenvector of $P_n$ since  $\lambda\neq 1$,  and this contradicts our assumption. Therefore,   $m_{{\cal L}_{n-1}}(2\lambda) = m_{{\cal L}_n}(\lambda)$.

	When $\lambda=\frac{3}{2}$, $\psi_j$ can be determined by Eq.~\eqref{F3} and we find that Eq.~\eqref{F2} holds for any $\psi_i$ and $\psi_j'$, and thus $m_{{\cal L}_n}(\frac{3}{2})=N_{n-1}$.
\end{proof}
\\

\begin{lemma}\label{eig2}
	 Let $\lambda$ be any eigenvalue of ${\cal L}_{n-1}$ such that $\lambda\neq 2$. Then $\frac{\lambda}{2}$ is an eigenvalue of ${\cal L}_n$. Besides,
	$m_{{\cal L}_n}\left(\frac{\lambda}{2}\right)=m_{{\cal L}_{n-1}}(\lambda)$.
\end{lemma}
\begin{proof}
	 This is a direct consequence of  Lemma~\ref{eig1}
\end{proof}
\\

\begin{remark}\label{remark1}
	Lemmas~\ref{eig1} and~\ref{eig2} show that a large fraction of the eigenvalues of ${\cal L}_n$ are obtained from the spectrum  of ${\cal L}_{n-1}$ as  each eigenvalue, except $2$, of ${\cal L}_{n-1}$ will generate a unique eigenvalue of ${\cal L}_n$. As we know the multiplicity of $\frac{3}{2}$, we conclude that we have  $(2N_{n-1}-m_{{\cal L}_{n-1}}(2))$ eigenvalues of $\sigma_n$. The rest of the spectrum consists of $1$s.
\end{remark}

From the Perron-Frobenius theorem~\cite{friedland2013perron}  we know that the largest absolute value of the eigenvalues of $P_n$  is  1.   Since there exists a unique stationary distribution  for random walks on $\tau^n(G)$, the multiplicity of eigenvalue $1$ of $P_n$ is  $1$. Thus the multiplicity of the smallest eigenvalue of ${\cal L}_n$ is also 1, i.e.  $m_{{\cal L}_n}(0)=1$.

We wish to find the  multiplicity of  2, the largest eigenvalue of ${\cal L}_n$, in the context of Markov chains.
A Markov chain is aperiodic if and only if the smallest eigenvalue of the Markov matrix is different from $-1$.
Thus, the largest eigenvalue of ${\cal L}_n$ is not equal to $2$ if and only if random walks on $\tau^n(G)$ are aperiodic.
Since each edge of $\tau^n(G)$, $n>0$ , belongs to an odd-length cycle (a  triangle), the graph is aperiodic~\cite{Ti96,DiSt91}.
Thus $m_{{\cal L}_n}(2)=0$ holds for  $n>0$.
Note, however, that the value of $m_{{\cal L}_0}(2)$ depends  on the structure of the initial graph $G$, since it may be periodic~\cite{Zh04}.

\begin{definition}
	Let $U=\{u_1,u_2,\cdots,u_k\}$ be any finite multiset of real number. The multiset $\mathcal{R}^{-1 }(U)$  is defined as
	\begin{equation}
			\mathcal{R}^{-1}(U)=\bigg\{\frac{u_1}{2},\frac{u_2}{2},\cdots,\frac{u_k}{2}\bigg\}.
	\end{equation}
\end{definition}

\begin{theo}\label{theo1}
The spectrum $\sigma_n$ of ${\cal L}_n$ is
 \begin{equation}
\sigma_n=
\begin{cases}
\mathcal{R}^{-1}\left(\sigma_{0}\backslash\{2\}\right)\cup\bigg\{\underbrace{\frac{3}{2}, \cdots, \frac{3}{2}}_\text{$N_{0}$}\bigg\} \cup\{\underbrace{1,1, \cdots, 1}_\text{$r_1+m_{{\cal L}_0}(2)$} \}   &\mbox{$n=1$}\\
\mathcal{R}^{-1}\left(\sigma_{n-1}\right)\cup\bigg\{\underbrace{\frac{3}{2}, \cdots, \frac{3}{2}}_\text{$N_{n-1}$}\bigg\} \cup\{\underbrace{1, 1,\cdots, 1}_\text{$r_n$}\} &\mbox{$n>1$},
\end{cases}
\end{equation}
where
\begin{equation}
r_n=\frac{3^{n-1}+1}{2}E_0-N_0.
\end{equation}
\end{theo}

\begin{proof}
From  results obtained  so far, the multiplicity of the eigenvalue $1$ of ${\cal L}_n$ can be determined indirectly:
\begin{equation}
\begin{split}
			m_{{\cal L}_n}(1)&=N_n-(2N_{n-1}-m_{{\cal L}_{n-1}}(2)) \\
			&=\frac{3^{n-1}+1}{2}E_0-N_0+m_{{\cal L}_{n-1}}(2).
\end{split}
\end{equation}
Since $m_{{\cal L}_{n-1}}(2)=0$ for $n>1$, the theorem is true.
\end{proof}

The two particular eigenvalues $1$ and $\frac{3}{2}$ of ${\cal L}_n$ are called exceptional eigenvalues for the family of matrices $\{{\cal L}_n\}$, for which the spectra shows self-similarity characteristics~\cite{bajorin2008vibration, teplyaev1998spectral}.

Theorem~\ref{theo1} confirms a previous result on the spectrum of the scale-free pseudofractal graph, obtained with other related methods by one of the authors~\cite{ZhLiGu15}.

\section{Applications of the normalized Laplacian spectra of an $n$-triangulation of a graph} \label{app}

From  the spectrum of ${\cal L}_n$, the normalized  Laplacian  of the $n$-triangulation graph $\tau^n(G)$, we compute in this section some relevant invariants related to its structure. We give closed formulas for the multiplicative degree-Kirchhoff index, Kemeny's constant and the number of spanning trees of $\tau^n(G)$. The results depend only on $n$ and some invariants of graph $G$.

\subsection{Multiplicative degree-Kirchhoff index}
\begin{theorem}\label{stan}
	For a simple connected graph $G$, the multiplicative degree-Kirchhoff indices of  $\tau^n(G)$ and  $\tau^{n-1}(G)$, $n>0$, are related as follows:
	\begin{equation}\label{K}
			{K\! f}^*(\tau^n(G))=6{K\! f}^*(\tau^{n-1}(G))-2\cdot3^{n-1}N_0E_0+(5\cdot3^{2n-2}+3^{n-1})E_0^2.
	\end{equation}
	Thus, the general expression for ${K\! f}^*(\tau^n(G))$ is
	\begin{equation}
		{K\! f}^*(\tau^n(G))=6^n{K\! f}^*(G)+(2\cdot3^{n-1}-4\cdot6^{n-1})N_0E_0+(5\cdot3^{2n-1}-8\cdot6^{n-1}-3^{n-1})E_0^2.
	\end{equation}
\end{theorem}
\begin{proof}
	 We use Eq.~\eqref{kirchhoff} and Theorem~\ref{theo1}:
\begin{equation}\label{K1}
\begin{split}
	{K\! f}^*(\tau^n(G))&=2E_n\left(\sum_{k=2}^{N_{n-1}}\frac{1}{\lambda_k^{(n-1)}/2}+\frac{1}{3/2}\times N_{n-1}+\frac{1}{1}\times r_n\right)\\
	&=2E_n\left(\frac{{K\! f}^*(\tau^{n-1}(G))}{E_{n-1}}+\frac{2N_{n-1}}{3}+r_n\right)\\
	&=6{K\! f}^*(\tau^{n-1}(G))-2\cdot3^{n-1}N_0E_0+(5\cdot3^{2n-2}+3^{n-1})E_0^2,
\end{split}
	\end{equation}
 where $\lambda_k^{(n)}$ are eigenvalues of  ${\cal L}_n$.  Finally:
\begin{equation}
	{K\! f}^*(\tau^n(G))=6^n{K\! f}^*(G)+(2\cdot3^{n-1}-4\cdot6^{n-1})N_0E_0+(5\cdot3^{2n-1}-8\cdot6^{n-1}-3^{n-1})E_0^2.
\end{equation}
\end{proof}

\subsection{Kemeny's constant}
\begin{theorem}\label{kemny}
	The Kemeny's constant for random walks on $\tau^n(G)$ can be obtained from $K(\tau^{n-1}(G))$ through
	\begin{equation}
			K(\tau^n(G))=2K(\tau^{n-1}(G))-\frac{N_0}{3}+\frac{5^{n-1}+1}{6}E_0.
	\end{equation}
	The general expression is
	\begin{equation}\label{kemeny}
			K(\tau^n(G))=2^nK(G)+\frac{1-2^n}{3}N_0+\frac{5\cdot3^n-2^{n+2}-1}{6}E_0.
	\end{equation}
\end{theorem}
\begin{proof}
	This result is a direct consequence of Theorem~\ref{stan} and Eq.~\eqref{KH}.
\end{proof}

\subsection{Spanning trees}

\begin{theorem}\label{sptree}
The number of spanning trees of $\tau^n(G)$, $n>0$, is:
\begin{equation}\label{spp}
N_{\rm st}^{(n)}=3^{\kappa-n}2^{\kappa-n(2N_0-E_0-1)}N_{\rm st}^{(0)}
\end{equation}
where $\kappa=\sum_{i=0}^{n-1}N_i=\frac{3^n-1}{4}E_0+n\left(N_0-\frac{E_0}{2}\right)$.
\end{theorem}

\begin{proof}
From Eq.~\eqref{ST01} and the definition of triangulation of a graph:
\begin{equation}\label{st1}
\frac{N_{\rm st}^{(n)}}{N_{\rm st}^{(n-1)}}=\frac{2^{N_n}}{3}\cdot\frac{\prod\limits_{i=2}^{N_n}\lambda_i^{(n)}}{\prod\limits_{i=2}^{N_{n-1}}\lambda_i^{(n-1)}},
\end{equation}
where $\lambda_i^{(n)}$ are the eigenvalues of $L_n$.  We obtain, for $n>0$:
\begin{equation}\label{st2}
\begin{split}
\prod\limits_{i=2}^{N_n}\lambda_i^{(n)}&=\left(\frac{3}{2}\right)^{N_{n-1}}\cdot\prod\limits_{i=2}^{N_{n-1}}  \frac{\lambda_i^{(n-1)}}{2}\\
&=\frac{3^{N_{n-1}}}{2^{2N_{n-1}-1}}\prod\limits_{i=2}^{N_{n-1}}\lambda_i^{(n-1)}.
\end{split}
\end{equation}

Therefore, the following equality
\begin{equation}\label{st4}
N_{\rm st}^{(n)}=\frac{3^{N_{n-1}-1}}{2^{2N_{n-1}-N_n-1}}\cdot N_{\rm st}^{(n-1)}=3^{N_{n-1}-1}2^{N_{n-1}-2N_0+E_0+1}N_{\rm st}^{(n-1)},
\end{equation}
holds for any $n>0$, and  finally we have:
\begin{equation}
N_{\rm st}^{(n)}=3^{\kappa-n}2^{\kappa-n(2N_0-E_0-1)}N_{\rm st}^{(0)},
\end{equation}
where $\kappa=\sum_{i=0}^{n-1}N_i$.
\end{proof}

We recall that former expressions for the multiplicative degree-Kirchhoff index, Kemeny's constant and the total number of spanning trees of an $n$-triangulation graph are valid given any initial simple connected graph $G$.

We  note here that the multiplicative degree-Kirchhoff index of $\tau^n(G)$ has been obtained  recently by Yang and Klein~\cite{YaKl15} by using a counting methodology not related with spectral techniques and also that
some results for the particular case of $k$-regular graphs have been  obtained in~\cite{HuLi15}  from the characteristic polynomials of the graphs.

Our result confirms both their calculations and the usefulness of the concise spectral methods described here.

Finally,  the values for Kemeny's constant and the number of spanning trees  of a scale-free pseudofractal graph, which is another particular  case of $n$-triangulation graph, were obtained by one of the authors in~\cite{ZhLiGu15}.
We find the same expressions by using $N_0=E_0=3$ in Eq.~\eqref{kemeny} and Eq.~\eqref{spp}.

\section{Conclusion}
	We have obtained a closed-form expression for the normalized Laplacian spectrum of a $n$-triangulation of any simple connected graph.   This was possible through the analysis of the eigenvectors in relation to adjacent vertices at different iteration steps.  Our method could be also applied to find the spectra of other  graph families which are constructed iteratively. We have also studied the spectrum structure.
 	The knowledge of this spectrum makes very easy, in relation to formerly known methods,  the calculation of several invariants related with structural and dynamic characteristics of the iterated triangulations of a graph. As an example, we find the multiplicative degree-Kirchhoff index,  Kemeny's constant and the number of spanning trees.

\section*{Acknowledgements}

This work was supported by the National Natural Science
Foundation of China under grant No. 11275049. F.C. was supported by the
Ministerio de Economia y Competitividad (MINECO), Spain, and the European
Regional Development Fund under project  MTM2011-28800-C02-01.




\end{document}